\documentclass[preprint,12pt]{elsarticle}
\usepackage{amssymb}
\usepackage[utf8]{inputenc}

\usepackage{amsthm}
\usepackage{amssymb,latexsym}
\usepackage{amsfonts}
\usepackage{amsmath}
\usepackage{enumerate}
\usepackage{float}
\usepackage{mathtools}
\usepackage{enumerate}
\usepackage{bbm}
\usepackage{bm}

\usepackage{booktabs}

\usepackage[usenames, dvipsnames]{xcolor}
\usepackage[colorlinks,linkcolor=blue,anchorcolor=blue,citecolor=blue]{hyperref}

\usepackage{graphicx}

\usepackage{caption}
\usepackage{subcaption}
\usepackage{comment}

\newcommand\R{{\mathbb R}}
\newcommand\N{{\mathbb N}}

\newcommand\al{\alpha}
\newcommand\be{\beta}
\newcommand\ga{\gamma}

\newcommand\la{\lambda}
\newcommand\ka{\kappa}

\newcommand\dist{\buildrel d \over =}

\newcommand{\M}{\mathcal{M}}
\newcommand{\T}{\mathbb{P}}
\newcommand{\E}[1]{\mathbb{E}{#1}}

\newcommand{\vphi}{\varphi}
\def\leq{\leqslant}
\def\geq{\geqslant}

\newtheorem{theorem}{Theorem}

\newtheorem{corollary}[theorem]{Corollary}
\newtheorem{lemma}[theorem]{Lemma}

\newtheorem{example}[theorem]{Example}

\usepackage{atbegshi}
\AtBeginDocument{\AtBeginShipoutNext{\AtBeginShipoutDiscard}}

\journal{arXiv}

\begin{document}

\begin{frontmatter}
\title{Distribution of Shifted Discrete Random Walk and Vandermonde matrices}

\author[inst1]{Andrius Grigutis}
\
\affiliation[inst1]{organization={Institute of Mathematics},
            addressline={Naugarduko 24}, 
            city={Vilnius},
            postcode={LT-03225}, 
            country={Lithuania}}


\begin{abstract}
In this work we set up the generating function of the ultimate time survival probability $\vphi(u+1)$, where $$\vphi(u)=\mathbb{P}\left(\sup_{n\geqslant 1}\sum_{i=1}^{n}\left(X_i-\ka\right)<u\right)$$
and $u\in\mathbb{N}_0,\,\ka\in\mathbb{N}$, and the random walk $\left\{\sum_{i=1}^{n}X_i,\,n\in\mathbb{N}\right\}$ consists of independent and identically distributed random variables $X_i$, which are non-negative and integer valued. We also give expressions of $\vphi(u)$ via the roots of certain polynomials. Based on the proven theoretical statements, we give several examples on $\vphi(u)$ and its generating function expressions, when random variables $X_i$ admit Bernoulli, Geometric and some other distributions.
\end{abstract}


\end{frontmatter}

\section{Introduction and preliminaries}\label{sec:introdution}

The study of sum of independent and identically distributed random va\-ri\-ab\-les (r.vs.) $\sum_{i=1}^{n}X_i$ is hardily avoidable in probability theory and related fields. This sequence of sums $\left\{\sum_{i=1}^{n}X_i,\,n\in\N\right\}$ is called {\it the random walk}. Let us define the stochastic process
\begin{align}\label{model}
W(n):=u+\ka n-\sum_{i=1}^{n}X_i,\,n\in\N,    
\end{align}
where $u\in\mathbb{N}_0:=\mathbb{N}\cup\{0\}$, $\kappa\in\mathbb{N}$ and random variables $X_i,\,i\in\mathbb{N}$ are independent, identically distributed, non-negative and integer valued. The defined process \eqref{model} is called {\it the  generalized premium discrete time risk model}, we abbreviate this naming by $GPDTRM$. Such type of processes appear in insurance mathematics arguing that they describe insurers wealth in time moments $n\in\N$, where $u$ means initial surplus (also called capital or reserve), $\ka$ denotes premium rate (earnings per unit of time), i.e. $(n+1)\ka-n\ka=\ka$, and the random walk $\left\{\sum_{i=1}^{n}X_i,\,n\in\N\right\}$ represents expenses caused by random size claims. Then, it is curious to know whether initial surplus and gained premiums are sufficient to cover an incurred random expenses. More precisely, one aims to know whether $W(n)>0$ for all $n\in\{1,\,2,\,\ldots,\,T\}$ when $T$ is some fixed natural number or $T\to\infty$. The positivity of $W(n)$ is of course associated to likelihood. For the $GPDTRM$ given in \eqref{model} we define {\it the finite time survival probability} 
$$
\vphi(u,T):=\T \left(\bigcap_{n=1}^{T}\left\{W(n)>0\right\}\right) = \T \left(\sup_{1\leq n \leq T}\sum_{i=1}^{n}\left(X_i-\ka\right)<u\right),\,T\in\N
$$
and {\it the ultimate time survival probability
\begin{align}\label{ult_prob}
\vphi(u):=\T \left(\bigcap_{n=1}^{\infty}\left\{W(n)>0\right\}\right) = \T \left(\sup_{n\geq 1}\sum_{i=1}^{n}\left(X_i-\ka\right)<u\right).
\end{align}
}
Both $\vphi(u,T)$ and $\vphi(u)$ are nothing but distribution functions of the provided integer valued sequence of sums of random variables; these functions are left-continuous, non-decreasing and step functions if we allow $u\in\R$. Also, $\vphi(\infty)=1$ if $\E X<\ka$, see the next Section \ref{sec:notations}. 

Calculation of $\vphi(u,T)$ is simple. If $X_i,\,i\in \N$ are independent copies of random variable (r.v.) $X$ and $x_i:=\T (X=i),\,i\in\N_0$, then
$$
\vphi(u,\,1)=\T(X\leq u+\ka-1),\,\vphi(u,\,T)=\sum_{i=1}^{u+\ka-1}\vphi(u+\ka-i,\,T-1)x_i,\,T\geq2.
$$
see, for instance, \cite[Theorem 1]{GS_sym}.

Let's turn to the ultimate time survival probability $\vphi(u)$. The law of total probability and rearrangements in \eqref{ult_prob} imply
\begin{align}\label{eq:recurrence}
\vphi(u)=\sum_{i=1}^{u+\kappa}x_{u+\kappa-i}\vphi(i),   
\end{align}
see \cite[page 3]{GS_sym}. 

By setting $u=0$ in \eqref{eq:recurrence}, we get
\begin{align}\label{u0}
\vphi(0)=x_{\ka-1}\vphi(1)+x_{\ka-2}\vphi(2)+\ldots+x_0\vphi(\ka),
\end{align}
what means that aiming to calculate $\vphi(\ka)$ when $x_0>0$, we must know know the initial ones $\vphi(0),\,\vphi(1),\,\ldots,\,\vphi(\ka-1)$. Equally, requirement to know $\vphi(0),\,\vphi(1),\,\ldots,\,\vphi(\ka-1)$ remains actual calculating $\vphi(u)$ for $u=\ka,\,\ka+1\,\ldots$ by recurrence \eqref{eq:recurrence}. The needed quantity of these initial values is $X$ distribution dependent as some of $x_0,\,x_1,\,\ldots,\,x_{\ka-1}$ may vanish, c.f. \eqref{u0} when $\T(X>j)=1$ for some $j\geq0$. The paper \cite{GS_sym} deals with finding the mentioned initial values and it is shown there that they can be found calculating limits of a certain recurrent sequences. For instance, if $\ka=2$ and $x_0>0$, then it follows by \eqref{u0} that 
$$
\vphi(0)=x_1\vphi(1)+x_0\vphi(2),
$$
where (see \cite[page 2 and 3]{GJ})

\begin{align}\label{limits}
\varphi(0) = \varphi(\infty)\lim_{n \to \infty}\frac{\gamma_{n+1} - \gamma_n}{\begin{vmatrix}\beta_n&\gamma_n\\\beta_{n+1}&\gamma_{n+1}\end{vmatrix}},
\,\varphi(1) = \varphi(\infty)\lim_{n \to \infty}\frac{\beta_n - \beta_{n+1}}{\begin{vmatrix}\beta_n&\gamma_n\\\beta_{n+1}&\gamma_{n+1}\end{vmatrix}},
\end{align}
when $|\cdot|$ is determinant,
\begin{align*}
&\be_0 = 1,\, \be_1 = 0, \, \be_n = \frac{1}{x_0} \left( \be_{n-2} -\sum_{i=1}^{n-1} x_{n-i} \be_i\right), \text{ for } n \geqslant 2,\\
&\ga_0= 0,\,\ga_1 = 1,\,  \ga_n = \frac{1}{x_0} \left( \ga_{n-2} -\sum_{i=1}^{n-1} x_{n-i} \ga_i\right), \text{ for } n \geqslant 2,
\end{align*}
and $\vphi(\infty)=1$ if $\E X<2$.

    Calculating the limits in \eqref{limits} and aiming to prove that provided determinant $2\times2$ never vanishes, in paper \cite{GJ} it was proved their connection to the solutions of $s^2=G_X(s)$, where $s\in\mathbb{C},\,|s|\leq1$ and $G_X(s)$ is the probability generating function of r.v. $X$. On top of that, it was realized in \cite{GJ} that the values of $\vphi(0)$ and $\vphi(1)$ in \eqref{limits} can be derived by the classical stationarity property for the distribution of the maximum of a reflected random walk, see \cite[Chapter VI, Section 9]{Feller}. Using the mentioned stationarity property, the generating function of $\vphi(u+1),\,u\in\mathbb{N}_0$ for $\ka=2$ was found in \cite[Theorem 5]{GJ}, however there was required the finitiness of the second moment of r.v. $X$, i.e. $\E X^2<\infty$. In this article, we proceed the work \cite{GJ} and find the generating function of $\vphi(u+1),\,u\in\mathbb{N}_0$ for arbitrary $\ka\in \N$. More over, we show that the requirement of $\E X^2<\infty$ is redundant and provide an exact expressions of $\vphi(u),\,u\in\N_0$ via solutions of systems of linear equations which are based on the roots of $s^\ka=G_X(s)$ and Vandermonde-like matrices.
    
For the short overview of literature, we mention that references \cite{Andersen}, \cite{Spitzer1}, \cite{Gerber}, \cite{Gerber1}, \cite{Shiu}, \cite{Shiu1}, \cite{DEVYLDER} are known as the classical ones on the wide subject of renewal risk models, while \cite{Rincon}, \cite{Cang} might be mentioned as the recent ones in nowadays. This work is also closely related to branching and Galton-Watson processes and queueing theory, see \cite{Kendal}, \cite{Kendal_1}, \cite{Kendal_2} and related papers. See also \cite{Arguin} or \cite[Figure 1]{Soundar} on random walks occurrence in number theory. Last but not least, it is worth mentioning that Vandermonde matrices have a broad range of occurrence from pure mathematics to many other applied sciences, see \cite{Rawashdeh} and related works. 

\section{Several auxiliary notations and the net profit condition}\label{sec:notations}

Let
\begin{align*}
&\mathcal{M}:=\sup_{n\geqslant1}\left(\sum_{i=1}^{n}(X_i-\kappa)\right)^+,
\end{align*}
where $x^+=\max\{0,x\}$, $x\in\mathbb{R}$ is the positive part function and r.vs. $X_i$ and $\ka\in\N$ are the same as in the model \eqref{model}. Let us denote the local probabilities of r.v. $\mathcal{M}$ by
\begin{align*}
\pi_i:=\mathbb{P}(\mathcal{M}=i),\, i\in\mathbb{N}_0
.
\end{align*}

Then, the ultimate time survival probability definition \eqref{ult_prob} implies that
\begin{align}\label{eq:phi_pi}
\varphi(u+1)=\sum_{i=0}^{u}\pi_i=\mathbb{P}(\mathcal{M}\leqslant u) \text{ for all }u\in\mathbb{N}_0.
\end{align}

In general, the r.v. $\mathcal{M}$ can be extended, i.e. $\T(\mathcal{M}=\infty)>0$, however the condition $\E X<\ka$
ensures $\T(\mathcal{M}<\infty)=1$. This is true due to
$$
\lim_{u\to\infty}\vphi(u)=1 \text{ if } \E X<\ka,
$$
see \cite[Lemma 1]{GS_sym}. The condition $\E X<\ka$ is called {\it the net profit condition} and it is crucial because the survival is impossible, i.e. $\vphi(u)=0$ for all $u\in\mathbb{N}_0$, if $\E X \geqslant \ka$, except few trivial cases when $\T(X=\ka)=1$, see \cite[Theorem 9]{GS_sym}.
Intuitively, it is clear that long term survival by model \eqref{model} is impossible if the threatening claim amount $X$ on average is equal or greater to the collected premium $\ka$ per unit of time.  

For $s\in\mathbb{C}$, let us denote the generating function  of $\vphi(1),\,\vphi(2),\,\ldots,$
$$
\Xi(s):=\sum_{i=0}^{\infty}\vphi(i+1)s^i,\,|s|< 1
$$
and the probability generating functions of r.vs. $X$ and $\mathcal{M}$
$$
G_X(s):=\sum_{i=0}^{\infty}x_is^i,\,
G_{\mathcal{M}}(s):=\sum_{i=0}^{\infty}\pi_is^i,\,
|s|\leq 1.
$$

Then, $\Xi(s)$ and $G_{\mathcal{M}}(s)$, for $|s|<1$, satisfy the relation
\begin{align}\label{relation_G}\nonumber
\Xi(s)&=\sum_{i=0}^{\infty}\vphi(i+1)s^i\\
&=\sum_{i=0}^{\infty}\sum_{j=0}^{i}\pi_js^i
=\sum_{j=0}^{\infty}\pi_j\sum_{i=j}^{\infty}s^i
=\frac{\sum_{j=0}^{\infty}\pi_js^j}{1-s}
=\frac{G_{\mathcal{M}}(s)}{1-s}.
\end{align}

In many examples, the radius of convergence of $G_X(s)$ or $G_{\mathcal{M}}(s)$ is larger than one. See \cite[Lemma 8]{GJ} for more properties of probability generating function in $|s|\leq1$.


\section{Main results}\label{sec:results}
In this section, based on the previously introduced notations and relation $(1-s)\Xi(s)=G_{\mathcal{M}}(s)$ in \eqref{relation_G}, we formulate the main results of the work.

\begin{theorem}\label{thm:main_eq}
Let's consider the GPDTRM defined in \eqref{model} and suppose that the net profit condition $\E X<\ka$ holds. Then, the local probabilities of random variables $\mathcal{M}$ and $X$ satisfy the following two equalities:
\begin{align}
\label{eq:main_eq}
G_{\M}(s)(s^\kappa-G_X(s))&=\sum_{i=0}^{\kappa-1}\pi_i\sum_{j=0}^{\kappa-1-i}x_j(s^\kappa-s^{i+j}),\,|s|\leq1,\\
\label{eq:main_eq_mean}
\kappa-\mathbb{E}X&=\sum_{i=0}^{\kappa-1}\pi_i\sum_{j=0}^{\kappa-1-i}x_j(\kappa-i-j).
\end{align}
\end{theorem}
We prove Theorem \ref{thm:main_eq} in Section \ref{sec:proof_of_thm}. 

    Equality \eqref{eq:main_eq} implies the following relation among the local probabilities $\pi_0$, $\pi_1$, $\ldots$

\begin{corollary}\label{cor:pi}
Let $\pi_i=\T (\mathcal{M}=i),\,i\in\mathbb{N}_0$ and $F_X(u)=\sum_{i=0}^{u}x_i,\,u\in\N_0$ be the distribution function of r.v. X. Then, for $\ka\in \N$, the following equalities hold:
\begin{align}\label{eq:pi_relation}
&\pi_\ka x_0=\pi_0-\sum_{i=0}^{\ka-1}\pi_iF_X(\ka-i),\\\nonumber    
&\pi_nx_0=\pi_{n-\ka}-\sum_{i=0}^{\ka-1}\pi_ix_{n-i}, \,n=\ka+1,\,\ka+2,\,\ldots
\end{align}
\end{corollary}

We explain the implication of Corollary \ref{cor:pi} in Section \ref{sec:proof_of_thm}.

    Let's turn to the survival probabilities $\vphi(1),\,\vphi(2),\,\ldots$ generating function $\Xi(s)$. It is easy to see that equalities \eqref{relation_G} and \eqref{eq:main_eq} imply

\begin{align}\label{eq:gen_f}
\Xi(s)=\frac{\sum_{i=0}^{\kappa-1}\pi_i\sum_{j=0}^{\kappa-1-i}x_j(s^\kappa-s^{i+j})}{(1-s)(s^\kappa-G_X(s))}.
\end{align}
Therefore, the similar way as the recurrence \eqref{eq:recurrence} requires the initial values of $\vphi(0)$, $\vphi(1)$, $\ldots$, $\vphi(\ka-1)$, the generating function $\Xi(s)$ in \eqref{eq:gen_f} (the equality \eqref{eq:pi_relation} as well) requires $\pi_0,\,\pi_1,\,\ldots,\,\pi_{\ka-1}$, $\ka\in\N$. These local probabilities of $\mathcal{M}$ can be solved out from relations \eqref{eq:main_eq} and \eqref{eq:main_eq_mean} and this is achievable as provided in items {\bf (i)}-{\bf (iv)} below:

\bigskip

({\bf i}) {\it We can choose such $|s|\leq1$ that the left hand-side of \eqref{eq:main_eq} vanishes, i.e. the roots of $s^\ka=G_X(s)$.}

\bigskip

({\bf i.1})
{\it If the net profit condition $G_X'(1)=\E X<\ka$ holds and the greatest common divisor of powers of $s$ in $s^\ka=G_X(s)$ is one, there are exactly $\ka-1$ roots of $s^\ka=G_X(s)$ in $|s|<1$ counted with their multiplicities. This fact is implied by Rouch\'e's theorem and estimate $|G_{X}(s)|<|\lambda s^\ka|$ when $\lambda>1$ and $|s|=1$, which means that the both functions $\la s^\ka$ and $\la s^\ka-G_X(s)$ have $\kappa$ zeros in $|s|<1$. When $\la\to1^+$, there is always one root out of those $\ka$ in $|s|<1$ migrating to $s=1$ \textup{(}$s=1$ is always the root of $s^\ka=G_X(s)$\textup{)} and some to other boundary points $|s|=1$ \textup{(}roots of unity\textup{)} if the greatest common divisor of powers of $s$ in $s^\ka=G_X(s)$ is greater than one, see \cite[Chapter 10]{Rudin}, \cite[Remark 10]{GJS} and \cite[Section 4, Lemma 9 and 10 therein]{GJ}}.

\bigskip

({\bf ii}) {\it Let $\al\neq1$ be the root of $s^\ka=G_X(s)$ in $|s|\leq1$ and denote $\bm{\pi}:=(\pi_0,\,\pi_1,\,\ldots,\,\pi_{\ka-1})^T$, where $T$ denotes the transpose. Then, by \eqref{eq:main_eq} and
$$
\left(\al^j+\al^{j+1}+\ldots+\al^{\ka-1}\right)(\al-1)=
\al^\ka-\al^j,\,j\in\{0,\,1,\,\ldots,\,\ka-1\},
$$
it holds that
\begin{align*}
&0=\left(
\sum_{j=0}^{\kappa-1}x_j(\al^\kappa-\al^j),\,\sum_{j=0}^{\kappa-2}x_j(\al^\kappa-\al^{j+1}),\,\ldots,\,x_0(\al^\kappa-\al^{\kappa-1})
\right)
\bm{\pi}\\
&=\left(
\sum_{j=0}^{\kappa-1}x_j\sum_{i=j}^{\ka-1}\al^i,\,\sum_{j=0}^{\kappa-2}x_j\sum_{i=j+1}^{\ka-1}\al^i,\,\ldots,\,x_0\al^{\kappa-1}
\right)
\bm{\pi}\\
&=\left(
\sum_{j=0}^{\kappa-1}\al^jF_X(j),\,\sum_{j=0}^{\kappa-2}\al^{j+1}F_X(j),\,\ldots,\,\al^{\kappa-1}x_0
\right)
\bm{\pi}=
\sum_{i=0}^{\ka-1}\pi_i\sum_{j=0}^{\ka-1-i}\al^{j+i}F_X(j)
,
\end{align*}
where $F_X(u)=\T(X\leqslant u)=\sum_{i=0}^{u}x_i,\,u\in \mathbb{N}_0$ is the distribution function of r.v. $X$.}

\bigskip

({\bf iii}) {\it Let $\al_1,\,\ldots,\,\al_{\kappa-1}\neq1$ be the roots of $s^\kappa=G_X(s)$ in $|s|\leq1$. Then, by {\bf (i)}, {\bf (ii)} and \eqref{eq:main_eq_mean},
\begin{align}\label{eq:system}\nonumber
&\begin{pmatrix}
\sum_{j=0}^{\kappa-1}\al_1^jF_X(j)&\sum_{j=0}^{\kappa-2}\al_1^{j+1}F_X(j)&\ldots&\al_1^{\kappa-1}x_0\\
\vdots&\vdots&\ddots&\vdots\\
\sum_{j=0}^{\kappa-1}\al_{\ka-1}^jF_X(j)&\sum_{j=0}^{\kappa-2}\al_{\ka-1}^{j+1}F_X(j)&\ldots&\al_{\ka-1}^{\kappa-1}x_0\\
\sum_{j=0}^{\kappa-1}x_j(\kappa-j)&\sum_{j=0}^{\kappa-2}x_j(\kappa-j-1)&\ldots&x_0
\end{pmatrix}
\begin{pmatrix}
&\pi_0\\
&\vdots\\
&\pi_{\kappa-2}\\
&\pi_{\kappa-1}
\end{pmatrix}
\\
&\hspace{9cm}=
\begin{pmatrix}
&0\\
&\vdots\\
&0\\
&\kappa-\mathbb{E}X
\end{pmatrix}.
\end{align}

If $A\bm{\pi}=B$ denotes the system \eqref{eq:system}, $x_0>0$ and $\al_1,\,\al_2,\,\ldots,\,\al_{\ka-1}\neq1$ are the roots of multiplicity one, then, according to Lemma \ref{lem:non_singular}, the determinant $|A|\neq0$ and $\bm{\pi}=A^{-1}B$, where $A^{-1}$ is the inverse matrix of $A$.}

\bigskip

({\bf iv}) {\it Suppose the root $\al\neq1$ of $s^{\ka}=G_X(s)$ in $|s|\leq1$ is of multiplicity $l\in\{2,\,3,\,\ldots,\,\ka-1\},\,\ka\geq3$. Then, according to equality \eqref{eq:main_eq} in Theorem \ref{thm:main_eq} and {\bf (ii)}, derivatives 
\begin{align}\label{eq:deriv}
\frac{d^{m}}{ds^{m}}\left(\sum_{i=0}^{\kappa-1}\pi_i\sum_{j=0}^{\kappa-1-i}s^{j+i}F_X(j)\right)\Bigg|_{s=\al}=0
\text{ for all }
m\in\{0,\,1,\,\ldots,\,l-1\}
\end{align}
and, in order to avoid identical lines in matrix $A$, we can set up the modified system \eqref{eq:system} by replacing its lines (except the last one) by the corresponding derivatives \eqref{eq:deriv}. If $x_0>0$, such a modified main matrix $A$ remains non-singular, see Lemma \ref{lem:non_s_der}.}

\bigskip

{\sc Note 1:} {\it The condition $x_0>0$ does not loose generality. If $\T (X>j)=1$ for some $j\in\{0,\,1,\,\ldots,\,\ka-2\},\,\ka\geq2$ and the net profit condition remains valid, then there reduces the order of recurrence in \eqref{eq:recurrence} and consequently some terms in sums of \eqref{eq:main_eq} and \eqref{eq:main_eq_mean} vanish causing the corresponding adjustments in system \eqref{eq:system} or its modified version described in {\bf (iv)}. We then end up dividing by some $x_{j+1}$ instead of $x_0$ where needed, c.f. \eqref{u0} or \cite[Theorem 7]{GS_sym}. In addition, we observe that $\T (X>\ka-1)$ implies $\E X\geq \ka$. Also, the both sides of $s^\ka=G_X(s)$ can be canceled by some power of $s\neq0$ if $\T(X>j)=1$ for some $j\in\{0,\,1,\,\ldots,\,\ka-2\},\,\ka\geq2$.
}

\bigskip

We further denote by $|A|$ the determinant of the matrix $A$ where $M_{i,\,j}$, $i,\,j\in\{1,\,2,\,\ldots,\,\ka\}$, $\ka\in\N$ are its minors and the matrix $A$ is the main matrix in \eqref{eq:system} or its modification replacing the coefficients by derivatives as described in ({\bf iv}). 

The equality \eqref{eq:gen_f} and thoughts listed in ({\bf i})-({\bf iv}) allow to formulate the following statement.

\begin{theorem}\label{thm:gen}
Let $|s|<1$ and $s^\ka-G_X(s)\neq0$. If the net profit condition $\E X<\ka$ holds, then the survival probability $\vphi(u+1),\,u\in\mathbb{N}_0$ generating function is
\begin{align}\label{eq:gen_f_final}
\Xi(s)=
\frac{\ka-\E X}{G_X(s)-s^\ka}
\sum_{i=0}^{\ka-1}\tilde{\pi}_i\sum_{j=0}^{\kappa-1-i}s^{j+i}F_X(j),
\end{align}
where $\tilde{\pi}_i=\pi_i/(\ka-\E X)$,
\begin{align*}
\tilde{\pi}_0=\frac{(-1)^{\ka+1}M_{\ka,\,1}}{|A|},\,
\tilde{\pi}_1=\frac{(-1)^{\ka+2}M_{\ka,\,2}}{|A|},\,
\ldots,\,
\tilde{\pi}_{\ka-1}=\frac{M_{\ka,\,\ka}}{|A|},
\end{align*}
and the matrix $A$ is created as provided in \textup{({\bf i})-({\bf iv})}.

    More over, the initial values for recurrence \eqref{eq:recurrence}, including $\vphi(\ka)$, are
\begin{align*}
&\vphi(0)=\frac{\ka-\E X}{|A|}\sum_{i=1}^{\ka}(-1)^{\ka+i}M_{\ka,\,i}\, F_X(\ka-i),\\
&\vphi(u)=\frac{\ka-\E X}{|A|}\sum_{i=1}^{u}(-1)^{\ka+i}M_{\ka,\,i},\,
u=1,\,2,\,\ldots,\,\ka.\\
\end{align*}
\end{theorem}
We prove Theorem \ref{thm:gen} in Section \ref{sec:proof_of_thm}.

\bigskip

{\sc Note 2}: {\it We agree that for $\ka=1$ the matrix $A=(x_0)$, its determinant $|A|=x_0$ and the minor $M_{1,\,1}=1$. Recall that $x_0$ gets replaced by some $x_{j+1}$ if $\T(X>j)=1$ for some $j\in\{0,\,1,\,\ldots,\,\ka-2\},\,\ka\geq2$ and the net profit condition holds, see {\sc Note 1}.
}

The next statement provides possible expressions of $\tilde{\pi}_0,\,\tilde{\pi}_1,\,\ldots,\,\tilde{\pi}_{\ka-1}$ and $\vphi(0),\,\vphi(1),\,\ldots, \vphi(\ka)$, $\ka\in\N$.

\begin{theorem}\label{thm:pi_expr}
Suppose that $x_0>0$ and $\al_1,\,\al_2,\,\ldots,\,\al_{\ka-1}\neq1$ are the roots of multiplicity one of $s^\ka=G_X(s)$ in $|s|\leq1$. Then, the values $\tilde{\pi_i}=\pi_i/(\ka-\E X)$ for $i=0,\,1,\,\ldots,\,\ka-1$ admit the following representation: 
\begin{align*}
&\tilde{\pi}_0=\frac{1}{x_0}\prod_{j=1}^{\ka-1}\frac{\al_j}{\al_j-1},\\
&\tilde{\pi}_1=-\frac{
\sum_{1\leq  j_1<\ldots<j_{\ka-2}\leq\ka-1}
\al_{j_1}\cdots\al_{j_{\ka-2}}}{x_0\prod_{j=1}^{\ka-1}(\al_j-1)}
-\frac{F_X(1)}{x_0}\tilde{\pi}_0,\\
&\tilde{\pi}_2=\frac{
\sum_{1\leq  j_1<\ldots<j_{\ka-3}\leq\ka-1}
\al_{j_1}\cdots\al_{j_{\ka-3}}}{x_0\prod_{j=1}^{\ka-1}(\al_j-1)}
-\frac{F_X(2)}{x_0}\tilde{\pi}_0-\frac{F_X(1)}{x_0}\tilde{\pi}_1,\\
&\vdots\\
&\tilde{\pi}_{\ka-1}=\frac{(-1)^{\ka+1}}{x_0}\prod_{j=1}^{\ka-1}\frac{1}{\al_j-1}-\frac{1}{x_0}\sum_{i=0}^{\ka-2}\tilde{\pi}_i F_X(\ka-1-i),\,\ka\geq2
\end{align*}

and the initial values for recurrence \eqref{eq:recurrence}, including $\vphi(\ka)$, are:

\begin{align*}
&\tilde{\vphi}(0)=(-1)^{\ka+1}\prod_{j=1}^{\ka-1}\frac{1}{\al_j-1},
\,\,\,\tilde{\vphi}(1)=\frac{1}{x_0}\prod_{j=1}^{\ka-1}\frac{\al_j}{\al_j-1},\\
&\tilde{\vphi}(2)=-\frac{F_X(1)}{x_0}\tilde{\vphi}(1)+\prod_{j=1}^{\ka-1}\frac{1/x_0}{\al_j-1}\left(\prod_{j=1}^{\ka-1}\al_j-\sum_{1\leq  j_1<\ldots<j_{\ka-2}\leq\ka-1}
\al_{j_1}\cdots\al_{j_{\ka-2}}\right),\\
&\tilde{\vphi}(3)=-\frac{F_X(1)}{x_0}\tilde{\vphi}(2)-\frac{F_X(2)}{x_0}\tilde{\vphi}(1)+\prod_{j=1}^{\ka-1}\frac{1/x_0}{\al_j-1}\\
&\times\left(\prod_{j=1}^{\ka-1}\al_j-\sum_{1\leq  j_1<\ldots<j_{\ka-2}\leq\ka-1}
\al_{j_1}\cdots\al_{j_{\ka-2}}
+\sum_{1\leq  j_1<\ldots<j_{\ka-3}\leq\ka-1}
\al_{j_1}\cdots\al_{j_{\ka-3}}\right),\\
&\hspace{7cm}\vdots
\end{align*}
\begin{align*}
&\tilde{\vphi}(\ka)=-\frac{1}{x_0}\sum_{i=1}^{\ka-1}F_X(\ka-i)\tilde{\vphi}(i)
+\prod_{j=1}^{\ka-1}\frac{1/x_0}{\al_j-1}
\Bigg(\prod_{j=1}^{\ka-1}\al_j-\\
&\sum_{1\leq  j_1<\ldots<j_{\ka-2}\leq\ka-1}
\al_{j_1}\cdots\al_{j_{\ka-2}}
+\sum_{1\leq  j_1<\ldots<j_{\ka-3}\leq\ka-1}
\al_{j_1}\cdots\al_{j_{\ka-3}}+\ldots+(-1)^{\ka+1}\Bigg),
\end{align*}
$\ka\geq2$,
where 
$$
\tilde{\vphi}(i)=\frac{\vphi(i)}{(\ka-\E X)},\, i\in\{0,\,1,\,\ldots,\,\ka\}.
$$ 
\end{theorem}

We prove Theorem \ref{thm:pi_expr} in Section \ref{sec:proof_of_thm}.

\bigskip

{\sc Note 3:} We define $\prod_{j=1}^{0}(\cdot)= \sum_{1\leq j_1<j_0\leq\ldots}(\cdot)=1$ in Theorem \ref{thm:pi_expr}. 

\bigskip

In view of Theorems \ref{thm:gen} and \ref{thm:pi_expr} we give several separate expressions on $\Xi(s)$.

\begin{corollary}\label{cor:xi}
If $\kappa=1$, then
\begin{align*}
\Xi(s)=\frac{1-\E X}{G_X(s)-s}.  
\end{align*}
If $\kappa=2$ and $x_0>0$, then
\begin{align*}
\Xi(s)=\frac{2-\E X}{\al-1}\cdot\frac{\al-s}{G_X(s)-s^2},
\end{align*}
where $\al\in[-1,0)$ is the unique root of $G_X(s)=s^2$.\\
If $\kappa=2$ and $x_0=0,\,x_1>0$, then
\begin{align*}
\Xi(s)=\frac{2-\E X}{\tilde{G}_{X}(s)-s},
\end{align*}
where $\tilde{G}_{X}(s)=\sum_{i=0}^{\infty}x_{i+1}s^i,\, |s|\leq1$.
\end{corollary}

We explain the implication of Corollary \ref{cor:xi} in Section \ref{sec:proof_of_thm}.

\section{Lemmas}\label{sec:lemas}

In this section we formulate and prove several auxiliary statements needed to derive the main results stated in Section \ref{sec:results}.

\begin{lemma}\label{same_dist}
The random variable $$\mathcal{M}=\sup_{n\geqslant1}\left(\sum_{i=1}^{n}\left(X_i-\ka\right)\right)^+,$$
where $x^+=\max\{0,\,x\}$ is the positive part of $x\in\R$,
admits the following distribution property
$$(\mathcal{M}+X-\kappa)^+\dist \mathcal{M}.$$
\end{lemma}

\begin{proof}
The proof is straight forward according to the definition of $\mathcal{M}$ and basic properties of maximum. Indeed,
\begin{align*}
&(\mathcal{M}+X-\kappa)^+=\max\left\{0,\,\max\left\{0,\,\sup_{n\geqslant1}\sum_{i=1}^{n}\left(X_i-\ka\right)\right\}+X-\ka\right\}\\
&\dist\max\left\{0,\,\max\left\{X_1-\ka,\,\sup_{n\geqslant2}\sum_{i=1}^{n}\left(X_i-\ka\right)\right\}\right\}\\
&\dist\max\left\{0,\,\sup_{n\geqslant1}\sum_{i=1}^{n}\left(X_i-\ka\right)\right\}=\mathcal{M}.
\end{align*}
See also, \cite[Lemma 5.2]{GJS}, \cite[Lemma 25]{GJ} and \cite[page 198]{Feller}.
\end{proof}

\begin{lemma}\label{lem:non_singular}
Let $\al_1,\,\ldots,\,\al_{\ka-1}\neq1$ be the roots of multiplicity one of $s^\ka=G_X(s)$ in the region $|s|\leq1$ and suppose that the local probability $x_0$ is positive. Then, the determinant $|A|$ of the main matrix in \eqref{eq:system} is 
$$
|A|=\frac{x_0^{\kappa}}{(-1)^{\kappa+1}}\prod_{j=1}^{\kappa-1}(\al_j-1)\prod_{1\leqslant i<j\leqslant \ka-1}(\al_j-\al_i)\neq0.
$$

\end{lemma}
\begin{proof}
Let us calculate the determinant $|A|=$
$$
\begin{vmatrix}
\sum_{j=0}^{\kappa-1}\al_1^jF_X(j)&\sum_{j=0}^{\kappa-2}\al_1^{j+1}F_X(j)&\ldots&\al_1^{\ka-2}x_0+\al_1^{\ka-1}F_{X}(1)&\al_1^{\kappa-1}x_0\\
\vdots&\vdots&\ddots&\vdots&\vdots\\
\sum_{j=0}^{\kappa-1}\al_{\ka-1}^jF_X(j)&\sum_{j=0}^{\kappa-2}\al_{\ka-1}^{j+1}F_X(j)&\ldots&\al_{\ka-1}^{\ka-2}x_0+\al_{\ka-1}^{\ka-1}F_{X}(1)&\al_{\ka-1}^{\kappa-1}x_0\\
\sum_{j=0}^{\kappa-1}x_j(\kappa-j)&\sum_{j=0}^{\kappa-2}x_j(\kappa-j-1)&\ldots&2x_0+x_1&x_0
\end{vmatrix}.
$$
We first put forward $x_0$ form the last column. Then, multiplying the last column by $F_X(\ka-1),\,F_X(\ka-2),\,\ldots,\,F_X(1)$ respectively and subtracting it from the first, the second and etc. columns, we obtain
$$
|A|=x_0\begin{vmatrix}
\sum_{j=0}^{\kappa-2}\al_1^jF_X(j)&\sum_{j=0}^{\kappa-3}\al_1^{j+1}F_X(j)&\ldots&\al_1^{\kappa-2}x_0&\al_1^{\kappa-1}\\
\vdots&\vdots&\ddots&\vdots&\vdots\\
\sum_{j=0}^{\kappa-2}\al_{\kappa-1}^jF_X(j)&\sum_{j=0}^{\kappa-3}\al_{\ka-1}^{j+1}F_X(j)&\ldots&\al_{\ka-2}^{\kappa-2}x_0&\al_{\kappa-1}^{\kappa-1}\\
\sum_{j=0}^{\kappa-2}x_j(\kappa-j-1)&\sum_{j=0}^{\kappa-3}x_j(\kappa-j-2)&\ldots&x_0&1
\end{vmatrix}.
$$
Proceeding the similar with the penultimate column of the last determinant and so on and applying the basic determinant properties, we obtain that $|A|$ equals to
\begin{align*}
x_0^{\kappa}\begin{vmatrix}
1&\al_1&\ldots&\al_1^{\kappa-1}\\
\vdots&\vdots&\ddots&\vdots\\
1&\al_{\kappa-1}&\ldots&\al_{\kappa-1}^{\kappa-1}\\\\
1&1&\ldots&1
\end{vmatrix}
&=
\frac{x_0^{\kappa}}{(-1)^{\kappa+1}}
\begin{vmatrix}
\al_1-1&\al_1^2-1&\ldots&\al_1^{\kappa-1}-1\\
\al_2-1&\al_2^2-1&\ldots&\al_2^{\kappa-1}-1\\
\vdots&\vdots&\ddots&\vdots\\
\al_{\kappa-1}-1&\al_{\kappa-1}^2-1&\ldots&\al_{\kappa-1}^{\kappa-1}-1
\end{vmatrix}\\
&=\frac{x_0^{\kappa}}{(-1)^{\kappa+1}}\prod_{j=1}^{\kappa-1}(\al_j-1)
\begin{vmatrix}
1&\al_1&\ldots&\al_1^{\kappa-2}\\
1&\al_2&\ldots&\al_2^{\kappa-2}\\
\vdots&\vdots&\ddots&\vdots\\
1&\al_{\kappa-1}&\ldots&\al_{\kappa-1}^{\kappa-2}
\end{vmatrix}.
\end{align*}
The last determinant is nothing but the well known Vandermonde determinant, see for example \cite[Section 6.1]{On_Vondermonde}. Thus,
$$
|A|=\frac{x_0^{\kappa}}{(-1)^{\kappa+1}}\prod_{j=1}^{\kappa-1}(\al_j-1)\prod_{1\leqslant i<j\leqslant \ka-1}(\al_j-\al_i)\neq0
$$
because the roots $\al_1,\,\al_2,\,\ldots,\,\al_{\ka-1}$ are distinct and lie in the region $|s|\leq1$, $s\neq1$. Note that 
$$
\prod_{j=1}^{0}(\cdot)=\prod_{1\leqslant i<j\leqslant 0}(\cdot)=\prod_{1\leqslant i<j\leqslant 1}(\cdot)=1
$$
by definition.
\end{proof}

\begin{lemma}\label{lem:non_s_der}
Let $|s|\leq1$. Suppose some roots $\al_1,\,\ldots,\,\al_{\ka-1}\neq1$ of $G_X(s)=s^\ka$ are multiple and assume that the local probability $x_0$ is positive. Then, the modified main matrix in \eqref{eq:system}, replacing its lines (except the last one) by derivatives \eqref{eq:deriv}, remains non-singular.
\end{lemma}

\begin{proof}
In short, the statement follows because derivative is the linear mapping. More precisely, if $\al_1$ is of multiplicity two, let's say, then there exists such sufficiently close to zero $\delta\in\R\setminus\{0\}$ that the matrix with the replaced second line

\begin{align}\label{matrix_delta}
\begin{pmatrix}
\sum_{j=0}^{\kappa-1}\al_1^jF_X(j)&\sum_{j=0}^{\kappa-2}\al_1^{j+1}F_X(j)&\ldots&\al_1^{\kappa-1}x_0\\
\sum_{j=0}^{\kappa-1}(\al_1+\delta)^jF_X(j)&\sum_{j=0}^{\kappa-2}(\al_1+\delta)^{j+1}F_X(j)&\ldots&(\al_1+\delta)^{\kappa-1}x_0\\
\vdots&\vdots&\ddots&\vdots\\
\sum_{j=0}^{\kappa-1}\al_{\ka-1}^jF_X(j)&\sum_{j=0}^{\kappa-2}\al_{\ka-1}^{j+1}F_X(j)&\ldots&\al_{\ka-1}^{\kappa-1}x_0\\
\sum_{j=0}^{\kappa-1}x_j(\kappa-j)&\sum_{j=0}^{\kappa-2}x_j(\kappa-j-1)&\ldots&x_0
\end{pmatrix}
\end{align}
is non-singular, see the expression of determinant in Lemma \ref{lem:non_singular}. Then, subtracting the second line from the first in \eqref{matrix_delta}, dividing the first line by $\delta$ afterwards and letting $\delta\to0$, we get the desired line replacement by derivative. 

The proof is analogous for higher derivatives and/or more multiple roots.
\end{proof}

\section{Proofs of the main results}\label{sec:proof_of_thm}

In this section we prove the statements formulated in Section \ref{sec:results}. Let's start with the proof of Theorem \ref{thm:main_eq}. 

\begin{proof}[Proof of Theorem \ref{thm:main_eq}]
By Lemma \ref{same_dist} and the rule of total expectation
\begin{align*}
G_{\M}(s)&=\mathbb{E}s^{\left(\M+X-\kappa\right)^+} 
=\mathbb{E}\left(\mathbb{E}\left(s^{\left(\M+X-\kappa\right)^+}|\M\right)\right)\\
&=\sum_{i=0}^{\kappa-1}\pi_i\mathbb{E}s^{\left(i+X-\kappa\right)^+}
+s^{-\kappa}G_X(s)\sum_{i=\kappa}^{\infty}\pi_is^i\\
&=\sum_{i=0}^{\kappa-1}\pi_i\left(\mathbb{E}s^{\left(X+i-\kappa\right)^+}
-s^{i-\kappa}G_X(s)\right)+s^{-\kappa}G_X(s)G_{\M}(s),
\end{align*}
which implies equality (\ref{eq:main_eq})
\begin{align*}
G_{\M}(s)(s^\kappa-G_X(s))&=\sum_{i=0}^{\kappa-1}\pi_i
(\mathbb{E}s^{(X+i-\kappa)^++\kappa}-s^iG_X(s))\\
&=\sum_{i=0}^{\kappa-1}\pi_i\sum_{j=0}^{\kappa-1-i}x_j(s^\kappa-s^{i+j}).
\end{align*}

To prove the second equality (\ref{eq:main_eq_mean}) in Theorem \ref{thm:main_eq} we take $s$ derivative of both sides of the derived equality \eqref{eq:main_eq}

\begin{align*}
&S_1+S_2:=G'_{\M}(s)(s^\kappa-G_X(s))+G_{\M}(s)(\kappa s^{\kappa-1}-G'_X(s))\\
&=\sum_{i=0}^{\kappa-1}\pi_i\sum_{j=0}^{\kappa-1-i}x_j(\kappa s^{\kappa-1}-(i+j)s^{i+j-1})=:S_3.
\end{align*}

We now let $s\to1^-$ in the last equality. It is easy to see that
\begin{align*}
\lim_{s\to1^-}S_3=\sum_{i=0}^{\kappa-1}\pi_i\sum_{j=0}^{\kappa-1-i}x_j(\kappa-i-j)  
\end{align*}
and 
\begin{align*}
\lim_{s\to1^-}S_2=\ka-\E X,
\end{align*}
because the net profit condition $\E X<\ka$ holds. Before calculating $\lim_{s\to1^-}S_1$, we observe that $\E X^2=\infty \,\Leftrightarrow\, \mathbb{E} \mathcal{M} = \infty$ and $\E X^2<\infty \,\Leftrightarrow\, \mathbb{E} \mathcal{M} < \infty$, see \cite[Theorem 5 and 6]{Kiefer}. Therefore, the requirement $\E X^2<\infty$ implies $\lim_{s\to1^-}S_1=0$ immediately. However, $\lim_{s\to1^-}S_1=0$ in spite $\mathbb{E} \mathcal{M} = \infty$. Indeed, if $G'_{\mathcal{M}}(s)\to\infty$ as $s\to1^-$, then
$$
\lim_{s\to1^-}S_1=\lim_{s\to1^-}\frac{s^\ka-G_X(s)}{1/G'_{\mathcal{M}}(s)}=
\lim_{s\to1^-}\frac{\ka s^{\ka-1}-G'_X(s)}{-G''_{\mathcal{M}}(s)/\left(G'_{\mathcal{M}}(s)\right)^2},
$$
where
$$
\limsup_{s\to1^-}\frac{\left(G'_{\mathcal{M}}(s)\right)^2}{G''_{\mathcal{M}}(s)}\leq\frac{N}{N-1}\sum_{i=N}^{\infty}\pi_i
$$
for any $N\in\{2,\,3,\,\ldots\}$, see \cite[Lemma 5.5]{GJS}.
Thus, equality (\ref{eq:main_eq_mean}) follows and the theorem is proved.
\end{proof}

\begin{proof}[Proof of Corollary \ref{cor:pi}]
The $n$'th derivative of both sides of equality (\ref{eq:main_eq}) and $s\to0$ gives
\begin{align*}
\pi_n x_0=\pi_{n-\kappa}-\sum_{i=0}^{n-1}\pi_ix_{n-i}
-\sum_{i=0}^{\kappa-1}\pi_i\sum_{j=0}^{\kappa-1-i}x_j
\mathbbm{1}_{\{n=\kappa\}},\,n=\kappa,\,\kappa+1,\,\ldots
\end{align*}
or
\begin{align*}
\pi_n x_0=\pi_{n-\kappa}-\sum_{i=0}^{n-1}\pi_ix_{n-i}
,\,n=\kappa+1,\,\kappa+2,\,\ldots
\end{align*}
\end{proof}
\begin{proof}[Proof of Theorem \ref{thm:gen}]
For $s^\ka-G_X(s)\neq0$, equality \eqref{eq:gen_f} and division by $1-s$ (see ({\bf ii}) in Section \ref{sec:results}) imply
\begin{align*}
\Xi(s)&=\frac{\sum_{i=0}^{\ka-1}\pi_i\sum_{j=0}^{\ka-1-i}s^{j+i}F_X(j)}{G_X(s)-s^\ka}\\
&=\frac{1}{G_X(s)-s^\ka}
\left(
\sum_{j=0}^{\kappa-1}s^jF_X(j),\,\sum_{j=0}^{\kappa-2}s^{j+1}F_X(j),\,\ldots,\,s^{\kappa-1}x_0
\right)
\begin{pmatrix}
\pi_0\\\pi_1\\
\vdots\\
\pi_{\ka-1}
\end{pmatrix}.
\end{align*}

By system \eqref{eq:system}, including its modified version described in ({\bf iv}) in Section \ref{sec:results}, and recalled notations $\bm{\pi}=(\pi_0,\,\pi_1,\,\ldots,\,\pi_{\ka-1})^T$ and $\tilde{\pi}_i=\pi_i/(\ka-\E X)$, we obtain
\begin{align*}
\bm{\pi}
&=\frac{1}{|A|}
\begin{pmatrix}
M_{1,\,1}&-M_{1,\,2}&\ldots&(-1)^{1+\ka}M_{1,\,\ka}\\
-M_{2,\,1}&M_{2,\,2}&\ldots&(-1)^{2+\ka}M_{2,\,\ka}\\
\vdots&\vdots&\ddots&\vdots\\
(-1)^{\ka+1}M_{\ka,\,1}&(-1)^{\ka+2}M_{\ka,\,2}&\ldots&M_{\ka,\,\ka}
\end{pmatrix}^T
\begin{pmatrix}
0\\
\vdots\\
0\\
\ka-\E X
\end{pmatrix}\\
&=\frac{\ka-\E X}{|A|}
\begin{pmatrix}
(-1)^{\ka+1}M_{\ka,\,1}\\
(-1)^{\ka+2}M_{\ka,\,2}\\
\vdots\\
M_{\ka,\,\ka}
\end{pmatrix}
=(\ka-\E X)
\begin{pmatrix}
\tilde{\pi}_0\\
\tilde{\pi}_1\\
\vdots\\
\tilde{\pi}_{\ka-1}
\end{pmatrix}
.
\end{align*}

Thus, the expression of $\Xi(s)$ in \eqref{eq:gen_f_final} follows.

    The claimed equalities on $\vphi(u)$ for $u=1,\,\ldots,\,\ka$ are evident due to obtained expression of $\bm{\pi}$ and $\vphi(u+1)=\sum_{i=0}^{u}\pi_i,\,u\in\N_0$ provided in \eqref{eq:phi_pi}. It remains to observe that the recurrence \eqref{eq:recurrence} yields
$$
\vphi(0)=\sum_{i=1}^{\ka}x_{\ka-i}\vphi(i)=\sum_{i=0}^{\ka-1}\pi_iF_X(\ka-1-i).
$$
\end{proof}

\begin{proof}[Proof of Theorem \ref{thm:pi_expr}]
Let us recall the matrix
\begin{align*}
A=\begin{pmatrix}
\sum_{j=0}^{\kappa-1}\al_1^jF_X(j)&\sum_{j=0}^{\kappa-2}\al_1^{j+1}F_X(j)&\ldots&\al_1^{\kappa-1}x_0\\
\vdots&\vdots&\ddots&\vdots\\
\sum_{j=0}^{\kappa-1}\al_{\ka-1}^jF_X(j)&\sum_{j=0}^{\kappa-2}\al_{\ka-1}^{j+1}F_X(j)&\ldots&\al_{\ka-1}^{\kappa-1}x_0\\
\sum_{j=0}^{\kappa-1}x_j(\kappa-j)&\sum_{j=0}^{\kappa-2}x_j(\kappa-j-1)&\ldots&x_0
\end{pmatrix}.
\end{align*}
Its determinant, according to Lemma \ref{lem:non_singular},
$$
|A|=\frac{x_0^{\kappa}}{(-1)^{\kappa+1}}\prod_{j=1}^{\kappa-1}(\al_j-1)\prod_{1\leqslant i<j\leqslant \ka-1}(\al_j-\al_i)\neq0.
$$
We now calculate the minors $M_{\ka,\,1},\,M_{\ka,\,2},\,\ldots,\,M_{\ka,\,\ka}$ of $A$. 
Following the calculation of $|A|$ in the proof of Lemma \ref{lem:non_singular}, we get
\begin{align*}
M_{\ka,\,1}
&=\begin{vmatrix}
\sum_{j=0}^{\kappa-2}\al_1^{j+1}F_X(j)&\sum_{j=0}^{\kappa-3}\al_1^{j+2}F_X(j)&\ldots&\al_1^{\kappa-1}x_0\\
\vdots&\vdots&\ddots&\vdots\\
\sum_{j=0}^{\kappa-2}\al_{\kappa-1}^{j+1}F_X(j)&\sum_{j=0}^{\kappa-3}\al_{\ka-1}^{j+2}F_X(j)&\ldots&\al_{\ka-1}^{\kappa-1}x_0
\end{vmatrix}\\
&=
x_0^{\ka-1}\begin{vmatrix}
\al_1&\al_1^2&\ldots&\al_1^{\ka-1}\\
\vdots&\vdots&\ddots&\vdots\\
\al_{\ka-1}&\al_{\ka-1}^2&\ldots&\al_{\ka-1}^{\ka-1}
\end{vmatrix}
=x_0^{\ka-1}\prod_{i=1}^{\ka-1}\al_i\prod_{1\leq i<j\leq\ka-1}\left(\al_j-\al_i\right).
\end{align*}
Note that $M_{\ka,\,1}$ is defined for $\ka\geq1$ and $M_{1,\,1}=1$ by agreement. The next one
\begin{align*}
M_{\ka,\,2}
&=\begin{vmatrix}
\sum_{j=0}^{\kappa-1}\al_1^{j}F_X(j)&\sum_{j=0}^{\kappa-3}\al_1^{j+2}F_X(j)&\ldots&\al_1^{\kappa-1}x_0\\
\vdots&\vdots&\ddots&\vdots\\
\sum_{j=0}^{\kappa-1}\al_{\kappa-1}^{j}F_X(j)&\sum_{j=0}^{\kappa-3}\al_{\ka-1}^{j+2}F_X(j)&\ldots&\al_{\ka-1}^{\kappa-1}x_0
\end{vmatrix}\\
&=x_0^{\ka-2}
\begin{vmatrix}
x_0+\al_1F_X(1)&\al_1^2&\ldots&\al_1^{\ka-1}\\
\vdots&\vdots&\ddots&\vdots\\
x_0+\al_{\ka-1}F_X(1)&\al_{\ka-1}^2&\ldots&\al_{\ka-1}^{\ka-1}
\end{vmatrix}
\\
&=x_0^{\ka-1}\prod_{1\leq i<j\leq\ka-1}(\al_j-\al_i)
\sum_{1\leq  j_1<\ldots<j_{\ka-2}\leq\ka-1}
\al_{j_1}\cdots\al_{j_{\ka-2}}+\frac{F_X(1)}{x_0}M_{\ka,\,1}.
\end{align*}

Similarly as before, $M_{\ka,\,2}$ is defined for $\ka\geq2$ only and $M_{2,\,2}=x_0+F_X(1)\al$, where $\al\in[-1,\,0)$ is the unique root of $s^2=G_X(s)$, see explanation ({\bf i}) in Section \ref{sec:results} and \cite[Section 4 and Corollary 15 therein]{GJ}.

Proceeding,
\begin{align*}
&M_{\ka,\,3}=\\
&\begin{vmatrix}
\sum_{j=0}^{\kappa-1}\al_1^{j}F_X(j)&\sum_{j=0}^{\kappa-2}\al_1^{j+1}F_X(j)&\sum_{j=0}^{\kappa-4}\al_1^{j+3}F_X(j)&\ldots&\al_1^{\kappa-1}x_0\\
\vdots&\vdots&\vdots&\ddots&\vdots\\
\sum_{j=0}^{\kappa-1}\al_{\kappa-1}^{j}F_X(j)&\sum_{j=0}^{\kappa-2}\al_{\kappa-1}^{j+1}F_X(j)&\sum_{j=0}^{\kappa-4}\al_{\ka-1}^{j+3}F_X(j)&\ldots&\al_{\ka-1}^{\kappa-1}x_0
\end{vmatrix}\\
&=x_0^{\ka-2}
\begin{vmatrix}
x_0+\al_1^2F_X(2)&\al_1&\al_1^3&\ldots&\al_1^{\ka-1}\\
\vdots&\vdots&\vdots&\ddots&\vdots\\
x_0+\al_{\ka-1}^2F_X(2)&\al_{\ka-1}&\al_{\ka-1}^3&\ldots&\al_{\ka-1}^{\ka-1}
\end{vmatrix}+\frac{F_X(1)}{x_0}M_{\ka,\,2}\\
&=x_0^{\ka-1}\prod_{1\leq i<j\leq\ka-1}(\al_j-\al_i)
\sum_{1\leq  j_1<\ldots<j_{\ka-3}\leq\ka-1}
\al_{j_1}\cdots\al_{j_{\ka-3}}\\
&\hspace{6cm}-\frac{F_X(2)}{x_0}M_{\ka,\,1}
+\frac{F_X(1)}{x_0}M_{\ka,\,2},\,\ka\geq3
\end{align*}
and so on until the last minor
\begin{align*}
&M_{\ka,\,\ka}
=\begin{vmatrix}
\sum_{j=0}^{\kappa-1}\al_1^{j}F_X(j)&\sum_{j=0}^{\kappa-2}\al_1^{j+1}F_X(j)&\ldots&x_0\al_1^{\ka-2}+\al_1^{\kappa-1}F_X(1)\\
\vdots&\vdots&\ddots&\vdots\\
\sum_{j=0}^{\kappa-1}\al_{\kappa-1}^{j}F_X(j)&\sum_{j=0}^{\kappa-2}\al_{\ka-1}^{j+1}F_X(j)&\ldots&x_0\al_{\ka-1}^{\ka-2}+\al_{\ka-1}^{\kappa-1}F_X(1)
\end{vmatrix}
\end{align*}

\begin{align*}
&=
\begin{vmatrix}
\sum_{j=0}^{\kappa-1}\al_1^{j}F_X(j)&\sum_{j=0}^{\kappa-2}\al_1^{j+1}F_X(j)&\ldots&x_0\al_1^{\ka-2}\\
\vdots&\vdots&\ddots&\vdots\\
\sum_{j=0}^{\kappa-1}\al_{\kappa-1}^{j}F_X(j)&\sum_{j=0}^{\kappa-2}\al_{\ka-1}^{j+1}F_X(j)&\ldots&x_0\al_{\ka-1}^{\ka-2}
\end{vmatrix}
+\frac{F_X(1)}{x_0}M_{\ka,\,\ka-1}=\\
&x_0^{\ka-1}\prod_{1\leq i<j\leq\ka-1}(\al_j-\al_i)
+(-1)^\ka\frac{F_X(\ka-1)}{x_0}M_{\ka,\,1}+(-1)^{\ka+1}
\frac{F_X(\ka-2)}{x_0}M_{\ka,\,2}\\
&\hspace{5cm}+\ldots+(-1)^{2\ka-1}\frac{F_X(2)}{x_0}M_{\ka,\,\ka-2}+\frac{F_X(1)}{x_0}M_{\ka,\,\ka-1}.
\end{align*}
The statement on expressions of $\tilde{\pi}_0,\,\tilde{\pi}_1,\,\ldots,\,\tilde{\pi}_{\ka-1}$ follows dividing the obtained minors with proper sings by determinant $|A|$.

We now prove the claimed formulas of $\vphi(0),\,\vphi(1),\,\ldots,\,\vphi(\ka)$, $\ka\in\N$. By the recurrence \eqref{eq:recurrence} with $u=0$, $\vphi(u+1)=\sum_{i=0}^{u}\pi_i$, $u\in\N_0$ and already proved expression of $\pi_{\ka-1}, \,\ka\in\N$ in Theorem \ref{thm:pi_expr}

\begin{align*}
\vphi(0)=\sum_{i=0}^{\ka-1}\vphi(i+1)x_{\ka-1-i}=\sum_{i=0}^{\ka-1}\pi_iF_X(\ka-1-i)=\frac{\ka-\E X}{(-1)^{\ka+1}}\prod_{j=1}^{\ka-1}\frac{1}{\al_j-1}. 
\end{align*}

The formula for $\vphi(1)$ is evident because $\vphi(1)=\pi_0$, where the expression of $\pi_0$ is already proved in Theorem \ref{thm:pi_expr} too. The rest is clear calculating the sum $\vphi(u+1)=\sum_{i=0}^{u}\pi_i$, $u\in\N_0$, where $\pi_i$ are given in the first part of Theorem \ref{thm:pi_expr}.

\end{proof}

\begin{proof}[Proof of Corollary \ref{cor:xi}]
The provided $\Xi(s)$ expressions are implied by Theorem $\ref{thm:gen}$. Recall that $s^2=G_X(s),\,x_0>0$ has the unique real root $\al\in[-1,\,0)$. In addition, when $x_0>0$, then $\al=-1$ is the root of $s^2=G_X(s)$ iff $\T(X\in2\N_0)=1$, see \cite[Section 4 and Corollary 15 therein]{GJ} and description $({\bf i})$ in Section \ref{sec:results}.
\end{proof}

\section{Particular examples}

In this section we give several examples illustrating the theoretical statements obtained in Section \ref{sec:results}. Required numerical computations are performed by Wolfram Mathematica \cite{Mathematica}

\begin{example}
Suppose the random claim amount $X$ is Bernoulli distributed, i.e. $1-\mathbb{P}(X=0)=p=\mathbb{P}(X=1),\,0<p<1$. We find the ultimate time survival probability generating function $\Xi(s)$ and calculate $\vphi(u),\,u\in\mathbb{N}_0$.
\end{example}

In view of the first part of Corollary \ref{cor:xi} and recurrence \eqref{eq:recurrence}, it is trivial that $\Xi(s)=1/(1-s),\,|s|<1$ and $\vphi(0)=x_0\vphi(1)=1-p$, $\vphi(u)=1$, $u\in\mathbb{N}$. In other words, the ultimate time survival is guaranteed if initial surplus $u\in\mathbb{N}$ and maximal claim size is one in the model $u+n-\sum_{i=1}^{n}X_i$.

\begin{example}
Suppose the random claim amount $X$ is distributed geometrically with parameter $p\in(0,1)$, i.e. $\T(X=k)=p(1-p)^k$, $k=0,\,1,\,\ldots$ and premium rate equals two, i.e. $\ka=2$. We find the ultimate time survival probability generating function $\Xi(s)$ and calculate $\vphi(0)$ and $\vphi(1)$, when the net profit condition is satisfied $\E X<2$.
\end{example}

We start with an observation on the net profit condition $$
\E X=\frac{1-p}{p}<2 \quad \Leftrightarrow \quad \frac{1}{3}<p<1.
$$ 

Then, according to Theorem \ref{thm:main_eq} and description ({\bf i}) in Section \ref{sec:results},
$$
G_X(s)=\frac{p}{1-(1-p)s}=s^2 \Rightarrow
\al:=s=\frac{p-\sqrt{4p-3p^2}}{2(1-p)}\in(-1,0),
$$
when $1/3<p<1$ and, by Corollary \ref{cor:xi} with $\ka=2$ and $x_0=p>0$,
$$
\Xi(s)=\frac{(3p-1)(p-\sqrt{4p-3p^2})}{p(3p-2-\sqrt{4p-3p^2})}\cdot\frac{1-(1-p)s}{(1-s)s^2+(s^3-1)p},
\,\frac{1}{3}<p<1.
$$
For $\ka=2$, $u=0$ and $1/3<p<1$ the recurrence \eqref{eq:recurrence} or Theorem \ref{thm:pi_expr} yields
\begin{align*}
\vphi(0)&=x_1\vphi(1)+x_0\vphi(2)=(1-p)p\,\Xi(0)+p\,\Xi'(0)
=\frac{2-\E X}{1-\al}
\\
&=\frac{3p-2+\sqrt{4p-3p^2}}{2p},\\
\vphi(1)&=\Xi(0)=\frac{2-\E X}{x_0}\frac{\al}{\al-1}  
=\frac{3p-\sqrt{4p-3p^2}}{2p^2}.
\end{align*}

One may check that for $p=101/300$
\begin{align*}
&\vphi(0)=\frac{\sqrt{90597}-297}{202}=0.0197691\ldots,\\
&\vphi(1)=\frac{45450-150\sqrt{90597}}{10201}=0.0295066\ldots
\end{align*}
and that coincides with the approximate values of $\vphi(0)$ and $\vphi(1)$ in \cite[page 12]{GS_sym} obtained via recurrent sequences.

\begin{example}
Let $X$ attain the natural values only, i.e. $x_0=0$, $x_1>0$. Let $\kappa=2$ and assume that the net profit condition is satisfied $\E X<2$. We provide the ultimate time survival probability $\vphi(u)$ formulas for all $u\in\mathbb{N}_0$.  
\end{example}
Let us recall that
$$
\tilde{G}_X(s)=\sum_{i=0}^{\infty}x_{i+1}s^i,\,|s|\leq1.
$$
The recurrence \eqref{eq:recurrence} and Corollary \ref{cor:xi} for $x_0=0$ and $x_1>0$ implies
\begin{align*}
\vphi(0)&=x_1\vphi(1)=2-\E X, \, \vphi(1)=\Xi(0)=\frac{2-\E X}{x_1},\\ \vphi(u)&=\frac{2-\E X}{(u-1)!}\frac{d^{u-1}}{ds^{u-1}}\left(\frac{1}{\tilde{G}_X(s)-s}\right)\Bigg|_{s=0}\\
&=\frac{1}{x_1}\left(\vphi(u-1)-\sum_{i=1}^{u-1}x_{u-i+1}\vphi(i)\right),\,u\geqslant2, 
\end{align*}
which echoes and widens the statement of Theorem 3 in \cite{GS_sym} providing another method of $\vphi(u),\,u\geqslant2$ calculation.

\begin{example}
Suppose the random claim amount $X$ is distributed geometrically with parameter $p=101/300$, i.e. $\T(X=k)=p(1-p)^k$, $k=0,\,1,\,\ldots$ and premium rate equals three, i.e. $\ka=3$. We set up the ultimate time survival probability generating function $\Xi(s)$ and calculate or provide formulas for $\vphi(u),\,u\in\mathbb{N}_0$.
\end{example}

First, we observe that the net profit condition is satisfied $\E X=199/101<3$. We now follow the statement of Theorem \ref{thm:main_eq} and surrounding comments beneath it. Then, for $p=101/300$, the equation
\begin{align*}
G_X(s)=\frac{p}{1-(1-p)s}=s^3    
\end{align*}
has two complex conjugate solutions $\al_1:=-0.368094 + 0.522097\bm{i}$ and $\al_2:=-0.368094 - 0.522097\bm{i}$ inside the unit circle $|s|<1$. Then, by Theorem \ref{thm:gen},
$$
\Xi(s)=\frac{\sum_{i=0}^{2}\pi_i\sum_{j=0}^{2-i}s^{i+j}F_X(j)}{s^3-G_X(s)},
$$
where $(\pi_0,\,\pi_1,\,\pi_2)=(0.582072,\, 0.0818989,\,0.0658497)$ is the unique solution of
\begin{align*}
\begin{pmatrix}
x_0+F_X(1)\al_1+F_X(2)\al_1^2&x_0\al_1+F_X(1)\al_1^2&x_0\al_1^2\\
x_0+F_X(1)\al_2+F_X(2)\al_2^2&x_0\al_2+F_X(1)\al_2^2&x_0\al_2^2\\
3x_0+2x_1+x_2&2x_0+x_1&x_0
\end{pmatrix}
\begin{pmatrix}
\pi_0 \\ \pi_1 \\ \pi_2
\end{pmatrix}
=
\begin{pmatrix}
0 \\ 0 \\ 3-\E X
\end{pmatrix}
\end{align*}
with appropriate numerical characteristics of the provided distribution. Theorem \ref{thm:pi_expr} and recurrence \eqref{eq:recurrence} imply:
\begin{align*}
&\vphi(0)=\sum_{i=1}^{3}x_{3-i}\vphi(i)=\frac{3-\E X}{(1-\al_1)(1-\al_2)}=0.480212\ldots,\\
&\vphi(1)=\pi_0=\frac{(3-\E X)\al_1\al_2}{x_0(1-\al_1)(1-\al_2)}=0.582072\ldots,\\
&\vphi(2)=\pi_0+\pi_1=(3-\E X)\frac{x_0(\al_1+\al_2)+x_1\al_1\al_2}{x_0^2(\al_1-1)(1-\al_2)}=0.663971\ldots,\\
&\vphi(3)=\pi_0+\pi_1+\pi_2=0.729821\ldots,\\
&\text{where }\pi_2=\frac{3-\E X}{x_0}\left(\frac{1}{(\al_1-1)(\al_2-1)}-\sum_{i=0}^{1}\pi_i F_X(2-i)\right),\\
&\vphi(u)=\frac{1}{x_0}\left(\vphi(u-3)-\sum_{i=1}^{u-1}x_{u-i}\vphi(i)\right)=\frac{d^{u-1}}{d s^{u-1}}\frac{\Xi(s)}{(u-1)!}\Bigg|_{s=0},\,u\geqslant3.
\end{align*}
The provided values of $\vphi(0),\,\vphi(1),\,\vphi(2)$ and $\vphi(3)$ coincide with the ones given in \cite[page 14]{GS_sym}, where they are obtained approximately from a certain recurrent sequences.

\begin{example}
Let $x_0=0.128$, $x_1=0.576$, $x_2=0.264$, $x_3=0.032$, $\sum_{i=0}^{3}x_i=1$ and $\ka=3$. We set up the ultimate time survival probability generating function $\Xi(s)$ and calculate $\vphi(u),\,u\in\mathbb{N}_0$.
\end{example}
For the provided distribution $\E X=1.2<3$ and the equation 
$$
0.128+0.576s+0.264s^2+0.032s^2=s^3
$$
has one root $s=-4/11=:\al$ of multiplicity two. Then, according to Theorem \ref{thm:main_eq} and comments ({\bf i})-({\bf iv}) beneath it, we create the modified system replacing the second line by the corresponding derivatives

\begin{align*}
\begin{pmatrix}
x_0+F_X(1)\al+F_X(2)\al^2&x_0\al+F_X(1)\al^2&x_0\al^2\\
F_X(1)+2F_X(2)\al&x_0+2F_X(1)\al&2x_0\al\\
3x_0+2x_1+x_2&2x_0+x_1&x_0
\end{pmatrix}
\begin{pmatrix}
\pi_0 \\ \pi_1 \\ \pi_2
\end{pmatrix}
=
\begin{pmatrix}
0 \\ 0 \\ 3-\E X
\end{pmatrix},
\end{align*}
which implies $(\pi_0,\,\pi_1,\,\pi_2)=(1,\,0,\,0)$ and consequently
\begin{align*}
\vphi(0)=0.968,\,\vphi(u)=1,\,u\in\mathbb{N},\, \Xi(s)=\frac{1}{1-s},\, |s|<1.   
\end{align*}
One may observe that the obtained result is expected, because $u+3n-\sum_{i=1}^{n}X_i>0$ for all $n\in\mathbb{N}$ except when $u=0$ and $X_i$ attains the value 3.

\section{Acknowledgments}
Author appreciates any kind of constructive criticism on the paper and thanks to processor Jonas Šiaulys for detailed reading the first draft of the manuscript. 

\bibliographystyle{elsarticle-harv} 
\bibliography{cas-refs}

\begin{flushleft}
Andrius Grigutis\\
Institute of Mathematics\\
Faculty of Mathematics and Informatics, Vilnius University\\
Naugarduko 24, LT-03225 Vilnius, Lithuania\\
andrius.grigutis@mif.vu.lt
\end{flushleft}

\end{document}